\documentclass[12pt,a4paper,twoside,reqno]{amsart}
\usepackage{indentfirst}
\usepackage{amsmath}
\usepackage{amssymb}
\usepackage{t1enc}
\usepackage[latin2]{inputenc}
\usepackage{graphicx}

\newcommand*{\D}{\ensuremath{\mathcal D}}
\newcommand{\e}{\varepsilon}

\theoremstyle{plain}
\newtheorem{theorem}{Theorem}
\newtheorem*{theorem*}{Theorem}
\newtheorem{lemma}{Lemma}
\newtheorem{rem}{Remark}

\usepackage{color}

\newcommand{\be}{\begin{equation}}
\newcommand{\ee}{\end{equation}}
\begin{document}
\title{ Weighted quantile correlation test for the logistic family}
\author{Ferenc Balogh}
\address{SISSA (Scuola Internazionale Superiore di Studi Avanzati)}
\curraddr{via Bonomea, 265, Trieste, Italy}
\email{fbalogh@sissa.it}
\thanks{}
\author{\'Eva Krauczi}
\address{Department of Mathematics, College of Kecskem\'et}
\curraddr{Izs\'aki \'ut 10, Kecskem\'et, Hungary--6000}
\email{osztenyine.eva@gamf.kefo.hu}
\subjclass[2000]{62F05, 60G15, 62Q05}
\keywords{ Correlation test, Karhunen-Lo\`eve expansion, power study,
 simulation, test of Logistic distribution, Wasserstein distance.}
\date{\today}

\begin{abstract}
We summarize the results of investigating the
asymptotic behavior of the weighted quantile correlation tests 
for the location-scale family associated to the logistic distribution. 
Explicit representations of the limiting distribution are given in terms 
of integrals of weighted Brownian bridges or
alternatively as infinite series of independent Gaussian random variables. 
The power of this test and the test for the location logistic family
against some alternatives are demonstrated by numerical simulations.

\end{abstract}

\maketitle

%%%%%%%%%%%%%%%

\section{Introduction}
\label{sec:intro}

The logistic distribution
\be
\label{eq:log_dist}
G(x)=\frac{1}{1+e^{-x}} \qquad x \in \mathbb{R}\ ,
\ee
 as the logistic growth curve was introduced 
in the mid-nineteenth century by Verhulst in his population dynamics study \cite{ver}.
The first purely statistical interpretation of the logistic distribution was found 
by Gumbel \cite{gumbel} in 1944 who showed 
that it is the asymptotic distribution of the midrange of random samples 
from symmetric continuous distributions.
Balakrishnan devoted a book to the logistic distribution \cite{handbook}, 
including goodness-of-fit tests. The routine goodness-of-fit 
techniques are presented: chi-squared tests, EDF statistics and tests based on 
regression and correlation. More tests on assessing the fit to the logistic 
distribution may be found in Aguirre and Nikulin \cite{AN} and Meintanis \cite{Me}. 
The main goal of the present paper is to introduce weighted quantile correlation tests
 for the location and the location-scale logistic families, introduced below.

The quantile correlation test statistics for goodness-of-fit to a family
 of probability distributions based on the $L_2$-Wasserstein 
distance were introduced by del Barrio, Cuesta-Albertos,
Matr\'an and Rodr\'{\i}\-guez-Rodr\'{\i}guez in \cite{BCMR}, 
considering goodness-of-fit tests to the normal family, and
del Barrio, Cuesta-Albertos and Matr\'an in \cite{BCM}. 
The asymptotic distributions of the test statistics are expressed in terms of
the Karhunen-Lo\`eve expansion of some associated weighted Brownian bridges.

The use of weight functions in the test statistics were independently suggested
 by de Wet in \cite{W1} and \cite{W2} and by Cs\"org\H o in \cite{Cs2} and \cite{Cs1}. 
Cs\"org\H o and Szab\'o introduced the new tests for several families 
 of probability distributions in \cite{CsSz} and \cite{CsSz2}.
 
 In this paper we use the same technique to obtain limiting distributions 
 for the weighted quantile correlation tests for the logistic family,
 using the known Karhunen-Lo\`eve expansion of the stochastic process
 \be
 \label{weighted_bridge}
 Z(t) = \frac{1}{\sqrt{t(1-t)}}B(t)\qquad 0<t<1\ ,
 \ee
 where $B(t)$ is the Brownian bridge on $[0,1]$ (see \cite{AD}).
 
The paper is organized as follows: in Section \ref{sec:wqct} we introduce 
weighted quantile correlation test statistics in detail and recall 
some earlier results on their limiting distributions. 
Section \ref{sec:logistic_tests} specializes the location-scale statistic for the logistic family, 
and the asymptotic distribution of the test statistic is given Theorem \ref{thm:limiting_integral}.
In Section \ref{sec:series} a different representation of the above limiting distribution
is obtained in Theorem \ref{thm:limiting_series}, given in terms of an infinite series
of independent Gaussian random variables.
Section \ref{sec:sim} contains a simulation study to evaluate the power of 
this test and the test for the location logistic family.

%%%%%%%%%%%%%%%

\section{Weighted quantile correlation tests and their asymptotics}
\label{sec:wqct}
 Given a random sample $X_1,\ldots,X_n$ with common 
distribution function $F(x)=P(X \leq x)$ on the real line $\mathbb{R}$, 
with the pertaining order statistics $X_{1,n}\le X_{2,n}\le\cdots\le X_{n,n}$,
let 
\be
F_n(x) = \frac{1}{n}\sum_{k=1}^n I\{X_k\le x\}\ , \qquad x\in\mathbb{R}
\ee
be
the empirical distribution function and let
\be
Q_n(t) = X_{k,n} \mbox{ if } (k-1)/n<t\le k/n\ ,\qquad  k=1,2,\ldots,n
\ee
be the sample quantile function. For a given distribution function $G(x)$,
 and for $\theta\in \mathbb{R}$ and $\sigma>0$, let $G^\theta_\sigma(x)=G((x-\theta)/\sigma)$,
 $x\in\mathbb{R}$, and consider the location-scale and location families
\be
{\mathcal G}_{l,s}=\{G^\theta_\sigma : \theta\in \mathbb{R}, \sigma>0\} \ ,\qquad 
{\mathcal G}_{l}=\{G^\theta_1 : \theta\in \mathbb{R} \}\ .
\ee
Denote by $Q_G(t)=G^{-1}(t)=\inf\{x\in\mathbb{R}: G(x)\geq t\}, 0<t<1$, the quantile function of $G$. 
Consider a weight function $w\ :\ (0,1)\to [0,\infty)$ satisfying $\int_0^1 w(t)\mathrm{d}t=1$.
Assume that the weighted second moment 
\be
\mu_2(G,w):=\int_0^1 Q_G^2(t) w(t)dt=\int_{-\infty}^\infty x^2 w(G(x))\mathrm{d}G(x)<\infty
\ee 
and the corresponding first moment
\be
 \mu_1(G,w):=\int_0^1Q_G(t)w(t)\mathrm{d}t=\int_{-\infty}^\infty x w(G(x))\mathrm{d}G(x)
\ee
 are also finite, as well as the generated variance 
 \be
 \nu(G,w):=\mu_2(G,w)-\mu_1^2(G,w)>0\ .
\ee
The weighted $L_2$-Wasserstein distance with weight function $w$ of two distributions $F$ and $G$ can be defined as 
\be
{\mathcal W}_w(F,G):= \left[\int_0^1 \left(Q_F(t)-Q_G(t)\right)^2 w(t)\mathrm{d}t\right]^{\frac{1}{2}}\ .
\ee
Therefore the weighted $L_2$-Wasserstein distance 
 ${\mathcal W}_w(F,{\mathcal G}_{l})= \inf \{{\mathcal W}_w(F,G): G\in {\mathcal G}_{l} \}$
 between $F$ and the location family ${\mathcal G}_{l}$ is 
 \be
 {\mathcal W}^2_w(F,{\mathcal G}_{l})=\int_0^1 \left(Q_F(t)-Q_G(t)\right)^2 w(t)\mathrm{d}t-\left[\int_0^1 
\left(Q_F(t)-Q_G(t)\right) w(t)\mathrm{d}t\right]^2 
 \ee
 and the weighted $L_2$-Wasserstein distance
 ${\mathcal W}_w(F,{\mathcal G}_{l,s})= \inf \{{\mathcal W}_w(F,G): G\in {\mathcal G}_{l,s} \}$ 
 between $F$ and location-scale family ${\mathcal G}_{l,s}$, scaled to $F$ is 
 \be
 \frac{{\mathcal W}^2_w(F,{\mathcal G}_{l,s})}{\nu(F,w)}=
 1-\frac{\left[\int_0^1 Q_F(t) Q_G(t) w(t)\mathrm{d}t-\mu_1(F,w)\mu_1(G,w) \right]^2}{\nu (F,w) \nu (G,w) }, 
 \ee
 as derived in \cite{Cs1}.
 
Then the location- and scale-free test statistic for the null-hypothesis
$H_0: F\in{\mathcal G}_{l,s}$ is
{\setlength\arraycolsep{.13889em}
\begin{eqnarray}
 &  V_n & =1-\frac{\left[\int_0^1 Q_n(t) Q_G(t) w(t)\mathrm{d}t-\mu_1(G,w) \int_0^1 Q_n(t) w(t)
\mathrm{d}t \right]^2}{ \nu (G,w) \left[\int_0^1 Q_n^2(t) w(t)\mathrm{d}t-\left(\int_0^1 Q_n(t) 
w(t)\mathrm{d}t\right)^2\right]} \\
 & &=1-\frac{\left[\sum_{k=1}^n X_{k,n}\left\{ \int_\frac{k-1}{n}^\frac{k}{n} Q_G(t) w(t)\mathrm{d}t-
\mu_1(G,w) \int_\frac{k-1}{n}^\frac{k}{n} w(t)\mathrm{d}t\right\}\right]^2}{ \nu (G,w) 
\left[\sum_{k=1}^n X_{k,n}^2\int_\frac{k-1}{n}^\frac{k}{n} w(t)\mathrm{d}t-\left(\sum_{k=1}^n X_{k,n}
\int_\frac{k-1}{n}^\frac{k}{n} w(t)\mathrm{d}t\right)^2\right]}\ 
\end{eqnarray}}
and the location-free test statistic for the null-hypothesis
$H_0: F\in{\mathcal G}_{l}$ is
{\setlength\arraycolsep{.13889em}
\begin{eqnarray}
 &  W_n & =\int_0^1 \left\{Q_n(t)-Q_G(t)\right\}^2 w(t)\mathrm{d}t-\left[\int_0^1 
\left\{Q_n(t)-Q_G(t)\right\} w(t)\mathrm{d}t\right]^2  \\
 & &=\nu (G,w) + \sum_{k=1}^n X_{k,n}^2\int_\frac{k-1}{n}^\frac{k}{n} w(t)\mathrm{d}t-
\left[\sum_{k=1}^n X_{k,n}\int_\frac{k-1}{n}^\frac{k}{n} w(t)\mathrm{d}t\right]^2 \\
\nonumber && \quad -2\sum_{k=1}^n X_{k,n}\left\{\int_\frac{k-1}{n}^\frac{k}{n} Q_G(t)w(t)\mathrm{d}t
-\mu_1(G,w)\int_\frac{k-1}{n}^\frac{k}{n} w(t)\mathrm{d}t\right\}. 
\end{eqnarray}}
 
 Understanding asymptotic relations as $n\to \infty$ throughout
this note, the symbols $\smash{\mathop{\longrightarrow}\limits^{\mathcal D}}$ and 
$\smash{\mathop{\longrightarrow}\limits^{ P}}$  denote convergence 
in distribution and in probability, respectively.

For the endpoints of the support of $G$, introduce 
\be
-\infty\leq a_G=\sup\{x: G(x)=0\}<\inf\{x: G(X)=1\}=b_G\leq\infty\ .
\ee
Finally, for each $n\in{\mathbb{N}}$ let $Y_{1,n}\le Y_{2,n}\le\cdots\le Y_{n,n}$ denote 
the order statistics of a sample  $Y_1,Y_2,\ldots,Y_n$ from $G$. 

The asymptotic distributions of $V_n$ and $W_n$ are of main practical interest 
for the statistician; to determine the limiting behaviour of these statistics, 
below we use the following general result due to Cs\"org\H{o}.

\begin{theorem*}[Cs\"org\H{o} \cite{Cs1}] \label{cs}
Let $w(\cdot)$ be a nonnegative integrable function on the interval $(0,1)$, 
for which $\int_0^1 w(t)\mathrm{d}t=1$.
Suppose that $G$ has finite weighted second moment and 
that it is twice differentiable on the open interval $(a_G,b_G)$ 
such that  $g(x)=G'(x)>0$ for all $x\in (a_G,b_G)$. If the conditions
 \begin{align}
 \label{sup}
\sup_{0<t<1}\frac{t(1-t)|g'(Q_G(t))|}{g^2(Q_G(t))}&<\infty\ ,\\
\label{int}
\int_0^1\frac{t(1-t)}{g^2(Q_G(t))}w(t)\mathrm{d}t&<\infty\ ,
\end{align}
and
\begin{align}
 \label{stoc1}
& n \int_0^{\frac{1}{n+1}}\left[Y_{1,n}-Q_G(t)\right]^2 w(t)\mathrm{d}t\;
\smash{\mathop{\longrightarrow}\limits^{ P}}\; 0\,,\\
\label{stoc2}
& n \int_{\frac{n}{n+1}}^1\left[Y_{n,n}-Q_G(t)\right]^2 w(t)\mathrm{d}t\;
\smash{\mathop{\longrightarrow}\limits^{ P}}\; 0\,,
\end{align}
are satisfied, the following asymptotics are valid:\\
\begin{enumerate}
 \item 
 If $F$ belong to ${\mathcal G}_{l}$ generated by $G$, then
\be
n W_n\;\smash{\mathop{\longrightarrow}\limits^{\mathcal D}}\; W\,,
\ee
where 
\be
  W = \int_0^1\frac{B^2(t) }{g^2(Q_G(t))}w(t)\mathrm{d}t- 
  \left[\int_0^1 \frac{B(t) }{g(Q_G(t))}w(t)  \mathrm{d}t\right]^2\ ,
\ee
for the location family, and
\item 
 if $F$ belong to ${\mathcal G}_{l,s}$ generated by $G$, then
 \be
n V_n\;\smash{\mathop{\longrightarrow}\limits^{\mathcal D}}\; V\,,
\ee
where 
\begin{align}
 V&=\frac{1}{\nu (G,w)}\left\{ \int_0^1\frac{B^2(t) }{g^2(Q_G(t))}w(t)\mathrm{d}t- 
\left[\int_0^1 \frac{B(t) }{g(Q_G(t))}w(t) \mathrm{d}t\right]^2\right\}\\
\nonumber &\ \ \  -\left[\frac{1}{\nu(G,w)}\int_0^1 \frac{B(t) Q_G(t)}{g(Q_G(t))}w(t)\mathrm{d}t-
\frac{\mu_1(G,w)}{\nu(G,w)}\int_0^1\frac{B(t) }{g(Q_G(t))}w(t) \mathrm{d}t\right]^2\ ,
\end{align}
for the location-scale family.
\end{enumerate}
\end{theorem*}
This theorem will be used in the next section to establish the asymptotic distributions
of the test statistics specialized to the logistic families.

%%%%%%%%%%%%%%%

\section{Tests for the logistic families and their asymptotics}
\label{sec:logistic_tests}

Consider the logistic distribution function \eqref{eq:log_dist}
with density function 
\be
g(x)=\frac{e^{-x}}{(1+e^{-x})^2}\qquad 
 \ x\in\mathbb{R}\ ,
\ee
 and ${\mathcal G}_{l,s}$ denotes the logistic location-scale family
 and ${\mathcal G}_{l}$ denotes the logistic location family as defined above. 
The corresponding quantile function is 
\be
Q_G(u)=\ln\frac{u}{1-u}\qquad  0<u<1\ .
 \ee
For the logistic location family ${\mathcal G}_{l}$ de Wet suggested in \cite{W2} the use of the weight function
$w(t)=\frac{L_1'(Q_G(t))}{I_1} $ is obtained, where $L_1(x)=\frac{-g'(x)}{g(x)}$, 
 $x\in \mathbb{R}$ and $I_1=\int_\mathbb{R} L_1'(x) g(x)\mathrm{d}x$, which gives
\be
\label{weight}
w(t)=6t(1-t)\qquad 0<t<1\ .
\ee

Note that de Wet proposes different weight functions for the logistic location 
and logistic scale families. The goal of this paper is to assess the use of 
the weight function \eqref{weight} for the combined logistic location-scale family.

Using integration by substitution and $\int_0^1 \frac{\ln t}{t-1}\mathrm{d}t=\frac{\pi^2}{6}$, we obtain 
\begin{align}
\mu_1(G,w)&=\int_0^1 6t(1-t)\ln\left(\frac{t}{1-t}\right)\mathrm{d}t=0,\\
\mu_2(G,w)&=\int_0^1 6t(1-t)\ln^2\left(\frac{t}{1-t}\right)\mathrm{d}t=\frac{\pi^2}{3}-2\ .
\end{align}
The above introduced location-scale-free test statistic specializes to

\be  V_n =1-\frac{\left[\displaystyle\sum_{k=1}^n a_{k,n}X_{k,n} \right]^2}
{ \left(\displaystyle\frac{\pi^2}{3}-2\right) \left[\displaystyle\sum_{k=1}^n b_{k,n}X_{k,n}^2-
\left(\sum_{k=1}^n b_{k,n}X_{k,n}\right)^2\right]}\ ,
\ee
where the coefficients are given explicitly by
\begin{align}
\nonumber
a_{k,n}&=\int_\frac{k-1}{n}^\frac{k}{n} 6t(1-t)\ln\left(\frac{t}{1-t}\right)\mathrm{d}t\\
\nonumber
&=\frac{k^2(3n-2k)}{n^3}\ln\frac{k}{n-k}-\frac{(k-1)^2(3n-2k+2)}{n^3}\ln\frac{k-1}{n-k+1}\\
&\ +\ln\frac{n-k}{n-k+1}+\frac{1-2k}{n^2}+\frac{1}{n}\ ,\\
b_{k,n}&=\int_\frac{k-1}{n}^\frac{k}{n} 6t(1-t)\mathrm{d}t=\frac{3(2k-1)}{n^2}+\frac{2(-3k^2+3k-1)}{n^3}\ .
\end{align} 
\begin{rem}
Note that the location-free test statistic is
\be
   W_n  = \left(\frac{\pi^2}{3}-2\right)+ \sum_{k=1}^n b_{k,n}X_{k,n}^2
-\left[\sum_{k=1}^n b_{k,n}X_{k,n}\right]^2-2\sum_{k=1}^n a_{k,n}X_{k,n}\ ,
\ee
(see \cite{W2}).
\end{rem}
As a corollary to the asymptotic results from \cite{Cs1} 
we obtain the following limiting distribution of the test statistics $V_n$.

\begin{theorem}
\label{thm:limiting_integral}
If the distribution function $F$ of the sample belongs to 
the logistic location-scale family ${\mathcal G}_{l,s}$ then
the rescaled statistic $nV_n$ has the asymptotic distribution
\be
n V_n\;\smash{\mathop{\longrightarrow}\limits^{\mathcal D}}\; V\,,
\ee
where 
\begin{align}
\nonumber
V &=\frac{1}{\frac{\pi^2}{3}-2}\left\{ \int_0^1\frac{6 B^2(t) }{t(1-t)}\mathrm{d}t- 
\left[\int_0^1 6 B(t) \mathrm{d}t\right]^2\right\}\\
\label{eq:V_int}
&\ -\left[\frac{1}{\frac{\pi^2}{3}-2}\int_0^1 6 B(t) \ln \left(\frac{t}{1-t}\right)\mathrm{d}t\right]^2,
\end{align}
where the integrals exists with probability $1$ and $B(\cdot)$ denotes a standard Brownian bridge.
\end{theorem}
\begin{rem}
If $F\in {\mathcal G}_{l}$ the theorem of Cs\"org\H o gives
\be
n W_n\;\smash{\mathop{\longrightarrow}\limits^{\mathcal D}}\; W\,,
\ee
where 
\be
\label{eq:W_int}
 W = \int_0^1\frac{6 B^2(t) }{t(1-t)}\mathrm{d}t- \left[\int_0^1 6 B(t) \mathrm{d}t\right]^2\ ,
\ee
in agreement with \cite{W2}.
\end{rem}
To proceed with the proof of Theorem \ref{thm:limiting_integral}, we need the following lemma.
\begin{lemma} The asymptotic relation
\label{int_asymp}
\be
n\int_{0}^{\frac{1}{n+1}}\ln^k\left(n\frac{t}{1-t}\right)t(1-t)\mathrm{d}t =
 (-1)^k\frac{k!}{2^{k+1}}\frac{1}{n}+{\mathcal O}\left(\frac{1}{n^2}\right)
\ee
holds for all integers $k \geq 0$.
\end{lemma}
The proof of this statement is postponed to the Appendix.
\begin{proof}[Proof of Theorem \ref{thm:limiting_integral}]
In order to prove the above convergence results, we need to verify the conditions
\eqref{sup} -- \eqref{stoc2}. 
Since
\be
g(Q_G(t)) = t(1-t)\ ,\qquad g'(Q_G(t)) = t(1-t)(1-2t)\ ,
\ee
conditions \eqref{sup} and \eqref{int} are satisfied.To conclude the proof we need to show that 
\[
M_{1,n}:= n \int_0^{\frac{1}{n+1}}\left[Y_{1,n}-\ln\left(\frac{t}{1-t}\right)\right]^2 6t(1-t)\mathrm{d}t\;
\smash{\mathop{\longrightarrow}\limits^{ P}}\; 0\,,
\]  
and 
\[
M_{n,n}:= n \int_{\frac{n}{n+1}}^1\left[Y_{n,n}-\ln\left(\frac{t}{1-t}\right)\right]^2 6t(1-t)\mathrm{d}t\;
\smash{\mathop{\longrightarrow}\limits^{ P}}\; 0\,,
\] 
where $Y_{1,n}\le Y_{2,n}\le\cdots\le Y_{n,n}$ is the order statistics from $G$.

An elementary calculation shows that the sequence of random variables
\be
A_n:= Y_{1,n}+\ln n \quad n=1,2,\dots
\ee
converges in distribution to $Y^*$, where 
\be
P(Y^*\leq x)=1-e^{-e^x}\ .
\ee
Therefore $A_n$ is stochastically bounded.
Hence we obtain
\begin{align}
M_{1,n}&= n \int_0^{\frac{1}{n+1}} \left[Y_{1,n}+\ln n-\left(\ln n+\ln\frac{t}{1-t}\right)\right]^2 6t(1-t)\mathrm{d}t\\
\nonumber
&=6A_n^2 n\int_{0}^{\frac{1}{n+1}}t(1-t)dt-12A_nn\int_{0}^{\frac{1}{n+1}}\ln\left(n\frac{t}{1-t}\right)t(1-t)\mathrm{d}t\\
\nonumber
&\ +6n\int_{0}^{\frac{1}{n+1}}\ln^2\left(n\frac{t}{1-t}\right)t(1-t)\mathrm{d}t \\
&=A_n^2\left(\frac{3}{n}+{\mathcal O}\left(\frac{1}{n^2}\right)\right)+A_n\left(\frac{3}{n}+{\mathcal O}
\left(\frac{1}{n^2}\right)\right) +\frac{3}{2n}+{\mathcal O}\left(\frac{1}{n^2}\right)\ ,
\end{align}
by Lemma \ref{int_asymp}, which shows that $M_{1,n}$ converges to $0$ in probability.

Similarly, 
\be
B_n:= Y_{n,n}-\ln n\smash{\mathop{\longrightarrow}\limits^{\mathcal D}}\;Y^{**}\ ,
\ee
where $P(Y^{**}\leq x)=e^{-e^{-x}}$, 
thus $B_n$ is stochastically bounded.
Notice that by the substitution $t \to 1-t$\ ,
\[
n\int_{\frac{n}{n+1}}^{1}\ln^k\left(n\frac{1-t}{t}\right)t(1-t)\mathrm{d}t = 
(-1)^k n\int_{0}^{\frac{1}{n+1}}\ln^k\left(n\frac{t}{1-t}\right)t(1-t)\mathrm{d}t\ .
\]
As above, we have 
\begin{align}
M_{n,n}&= n \int_{\frac{n}{n+1}}^1 \left[Y_{n,n}-\ln n+\left(\ln n-\ln\frac{t}{1-t}\right)\right]^2 6t(1-t)\mathrm{d}t\\
&=B_n^2\left(\frac{3}{n}+{\mathcal O}\left(\frac{1}{n^2}\right)\right)-B_n\left(\frac{3}{n}+{\mathcal O}
\left(\frac{1}{n^2}\right)\right)
 +\frac{3}{2n}+{\mathcal O}\left(\frac{1}{n^2}\right)\ ,
\end{align}
and therefore $M_{n,n}$ converges to $0$ in probability also, that concludes the proof.
\end{proof}

%%%%%%%%%%%%%%%

\section{Infinite series representations of the limiting distributions}
\label{sec:series}

The integral representations \eqref{eq:V_int} of the limiting distribution 
suggest that the Karhunen-Lo\`eve expansion of the weighted Brownian bridge \eqref{weighted_bridge},
once calculated, can be used to obtain an infinite series representation 
of the asymptotic distribution (see e.g. \cite{AG}).

Note that the covariance function
\be
\label{eq:cov_kernel}
K(s,t):= \textrm{Cov} (Z(s),Z(t)) = \frac{\min(s,t)-s t}{\sqrt{t(1-t)s(1-s)}}
\ee
of the process $Z(t)$ belongs to $L^2\left((0,1)^2\right)$, but it is not continuous
on the closed unit square $[0,1]^{2}$. A suitable extension to the standard results on 
integral operators with continuous kernels is employed in \cite{AD}
to treat the integral kernel \eqref{eq:cov_kernel} as a special example 
(for a more general setting, compare with \cite{DM}). 
It is shown in \cite{AD} that the stochastic process
\be
\label{eq:psi_B}
Z(t) =\frac{1}{\sqrt{t(1-t)}}B(t) \qquad 0 < t <1
\ee
admits the Karhunen-Lo\`eve expansion
\be
\label{eq:KL_series} 
Z(t)=\sum_{k=1}^{\infty} \sqrt{\lambda_k}Z_k f_k(t) \ ,
 \ee
where the normalized eigenfunctions $f(t)=f_k(t)$ can be written in the form
\be
f(t) = \frac{y(t)}{\sqrt{t(1-t)}} \ ,
\ee
where $y(t)$ solves the differential equation
\be
\label{eq:de}
y''(t)+\frac{1}{\lambda}\frac{1}{t(1-t)}y(t)=0
\ee
with boundary conditions
\be
\label{eq:bc}
y(0)=0\quad \mbox{and}\quad y(1)=0\ ,
\ee
corresponding to the nonzero eigenvalues $\lambda=\lambda_k$ of the associated integral operator. 
The random coefficients are
\be
Z_k=\frac{1}{\sqrt{\lambda_k}}\int_0^1 \frac{B(t)}{\sqrt{t(1-t)}} f_k(t)\mathrm{d}t, \quad k=1,2\ldots\ .
\ee

By substituting $t=\frac{x+1}{2}$ and $u(x)=y(\frac{x+1}{2})$,  
the differential equation \eqref{eq:de} is brought to
\be
\label{eq:de_jacobi}
u''(x)+\frac{1}{\lambda}\frac{1}{1-x^2} u(x)=0 \qquad -1<x<1\ ,
\ee
 which is of the form of the  \emph{Jacobi equation} with parameters $\alpha=1$ and $\beta=1$
 , and the boundary conditions
\be
\label{eq:bc_jacobi}
 u(-1)=0 \quad \mbox{and}\quad u(1)=0
\ee
restrict the values of $\lambda$ to be
\be
\label{eq:eigenvalues}
\lambda_k=\frac{1}{k(k+1)}\qquad k=1,2,\dots
\ee
(see \cite{abramowitz_stegun}, 22.6.2).
The full set of solutions to the boundary-value problem \eqref{eq:de_jacobi}-\eqref{eq:bc_jacobi} 
can be written in terms of the \emph{Jacobi orthogonal polynomials}\footnote{
In \cite{AD}, the normalized eigenfunctions are written in terms of the 
\emph{Ferrer associated Legendre polynomials} $P_k^{1}(x)$ (see \cite{WW},
p.~323); for our purposes it is more convenient to express them through 
the Jacobi polynomials $P_k^{(1,1)}(x)$, as suggested in \cite{W73}.}
 $P_k^{(1,1)}(x)$ (see \cite{abramowitz_stegun}, 22.6.2). 
 Therefore the original boundary-value problem \eqref{eq:de}-\eqref{eq:bc} gives the eigenfunctions
\be
\label{eq:eigenfunctions}
f_k(t) =\sqrt{\frac{(2k+1)(k+1)}{k}}P_{k-1}^{(1,1)}(2t-1)\sqrt{t(1-t)} \quad k=1,2,\dots
\ee
associated to the eigenvalues \eqref{eq:eigenvalues}. The normalization in \eqref{eq:eigenfunctions} is chosen so that
\be
\int_{0}^{1}f_k(t)f_l(t)\mathrm{d}t = \delta_{kl} \qquad k,l =1,2,\dots
 \ee
(see \cite{abramowitz_stegun}, 22.2.1).

Since $Z(t)$ is a Gaussian process and 
\be
E(Z_n Z_m)=\delta_{n,m} \qquad n,m=1,2,\dots
\ee
it follows that the variables $Z_n$ are independent standard normal random variables.

Given the Karhunen-Lo\`eve expansion of the weighted Brownian bridge \eqref{weighted_bridge}, 
the integral representation \eqref{eq:V_int} of the limiting distribution $V$ possesses 
the following infinite series representation.

\begin{theorem} 
\label{thm:limiting_series}
The limiting distribution $V$ can be represented alternatively as
\be
\label{eq:V_series}
V \ \mathop{=}\limits^{\D}\ \frac{1}{\frac{\pi^2}{3}-2}\sum _{k=2}^\infty \frac{6}{k(k+1)}Z_m^2
\ -\left[\frac{1}{\frac{\pi^2}{3}-2}\sum_{l=1}^\infty \frac{3\sqrt{4l+1}}{l(l+1)(2l-1)(2l+1)}Z_{2l}\right]^2\ ,
\ee
where $\{Z_m\}_{m=1}^{\infty}$ is an infinite sequence of independent
identically distributed standard normal random variables, and the series converges with probability one.
\end{theorem}

\begin{rem}
The integral representation \eqref{eq:W_int} of $W$ for the logistic location family implies the series expansion
\be
W\ \mathop{=}\limits^{\D}\  \sum_{k=2}^\infty \frac{6}{k(k+1)}Z_k^2\ ,
\ee
where $\{Z_m\}_{m=1}^{\infty}$ is an infinite sequence of independent identically distributed 
standard normal random variables, and the series converges with probability one, as shown in \cite{W2}.
\end{rem}
We need the following lemma to determine the coefficients in the infinite series representations of $V$.
\begin{lemma} 
\label{jacobi_integrals}
The following formula is valid:
\be
\int_{-1}^1 P_{n}^{(1,1)}(x)(1-x^2)\ln\left(\frac{1+x}{1-x}\right)dx=
\left\{
\begin{array}{cc}
\displaystyle\frac{8}{(2k+1)(2k+3)(k+2)} & n=2k+1\\
\displaystyle 0 & n=2k\ .
\end{array}
\right.
\ee
\end{lemma}
Since the proof of this lemma is quite technical, it is left to the Appendix.

\begin{proof}[Proof of Theorem \ref{thm:limiting_series}]
The Karhunen-Lo\`eve expansion \eqref{eq:KL_series} can be used 
to evaluate the integrals in \eqref{eq:W_int} in terms of the random coefficients $Z_k$, as shown below.
\begin{align}
\int_0^1\frac{B^2(t) }{t(1-t)}\mathrm{d}t&=\sum_{k,l=1}^{\infty}\sqrt{\lambda_k\lambda_l}
\int_{0}^{1}Z_kZ_l f_k(t)f_l(t)\mathrm{d}t\\
&=\sum_{k=1}^{\infty}\frac{1}{k(k+1)}Z_k^2\ .
\end{align}
Since $f_1(t) = \sqrt{6}\sqrt{t(1-t)}$, we have
\begin{align}
\left[\int_0^1B(t)\mathrm{d}t\right]^2&=
\left[\sum_{k=1}^{\infty}\sqrt{\lambda_k}Z_k\int_{0}^{1}f_k(t)\sqrt{t(1-t)}\mathrm{d}t\right]^2\\
&=\frac{1}{6}\left[\sum_{k=1}^{\infty}\sqrt{\lambda_k}Z_k\int_{0}^{1}f_k(t)f_1(t)\mathrm{d}t\right]^2\\
&=\frac{1}{6}\lambda_1 Z_1^2\ .
\end{align}
Hence
\begin{align}
\int_{0}^{1}B(t) \ln \left(\frac{t}{1-t}\right)\mathrm{d}t &=\sum_{k=1}^{\infty} \sqrt{\lambda_k}Z_k \int_{0}^{1}f_k(t)\sqrt{t(1-t)}\ln \left(\frac{t}{1-t}\right)\mathrm{d}t\\
&=\sum_{k=1}^{\infty} \sqrt{\lambda_k}Z_k \sqrt{\frac{(2k+1)(k+1)}{k}}\frac{1}{8}\int_{-1}^{1}P_{k-1}^{(1,1)}(x)(1-x^2)\ln \left(\frac{1+x}{1-x}\right)\mathrm{d}x\\
&=\sum_{l=1}^{\infty} \frac{\sqrt{4l+1}}{2l(l+1)(2l-1)(2l+1)}Z_{2l}\ ,
\end{align}
where, in the last equality, we have used Lemma \ref{jacobi_integrals}.
Combining the above results with the appropriate constant coefficients, 
the infinite series representation \eqref{eq:V_series} for the distribution of $V$ follows.
\end{proof}

%%%%%%%%%%%%%%%

\section{Performance of tests}
\label{sec:sim}

\subsection{The asymptotic distributions and the distributions of the tests $nV_n$ and $nW_n$}

 The distribution functions of the limiting random variables above
 are computed numerically by simulation, using their infinite series representations. 
We generated $200\,000$ copies of the random variable
$W$ and $V$, and we computed numerically their empirical
distribution function $H_l$ and $H_{l,s}$, respectively, each time truncating the series at $10\,000$. 
The parameters were chosen such that the values
of $H_l$ and $H_{l,s}$ be the same to two decimal places for different samples, respectively. 
The asymptotic distributions are shown in Fig. 1.

\begin{figure}[!hbp]
\begin{minipage}[b]{0.45\linewidth} 
\centering
\includegraphics[width=7cm]{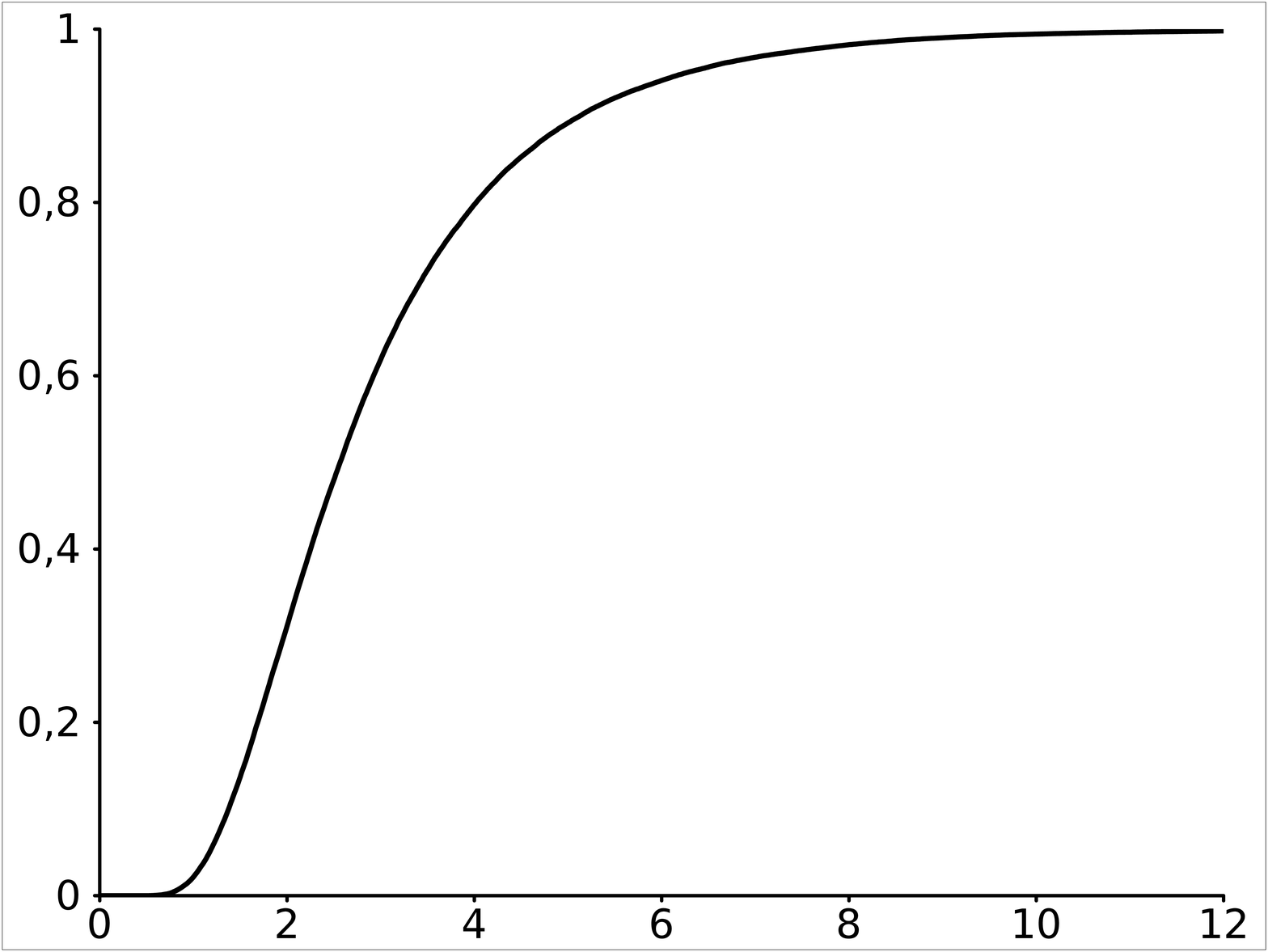}
\end{minipage}
\begin{minipage}[b]{0.45\linewidth} 
\centering
\includegraphics[width=7cm]{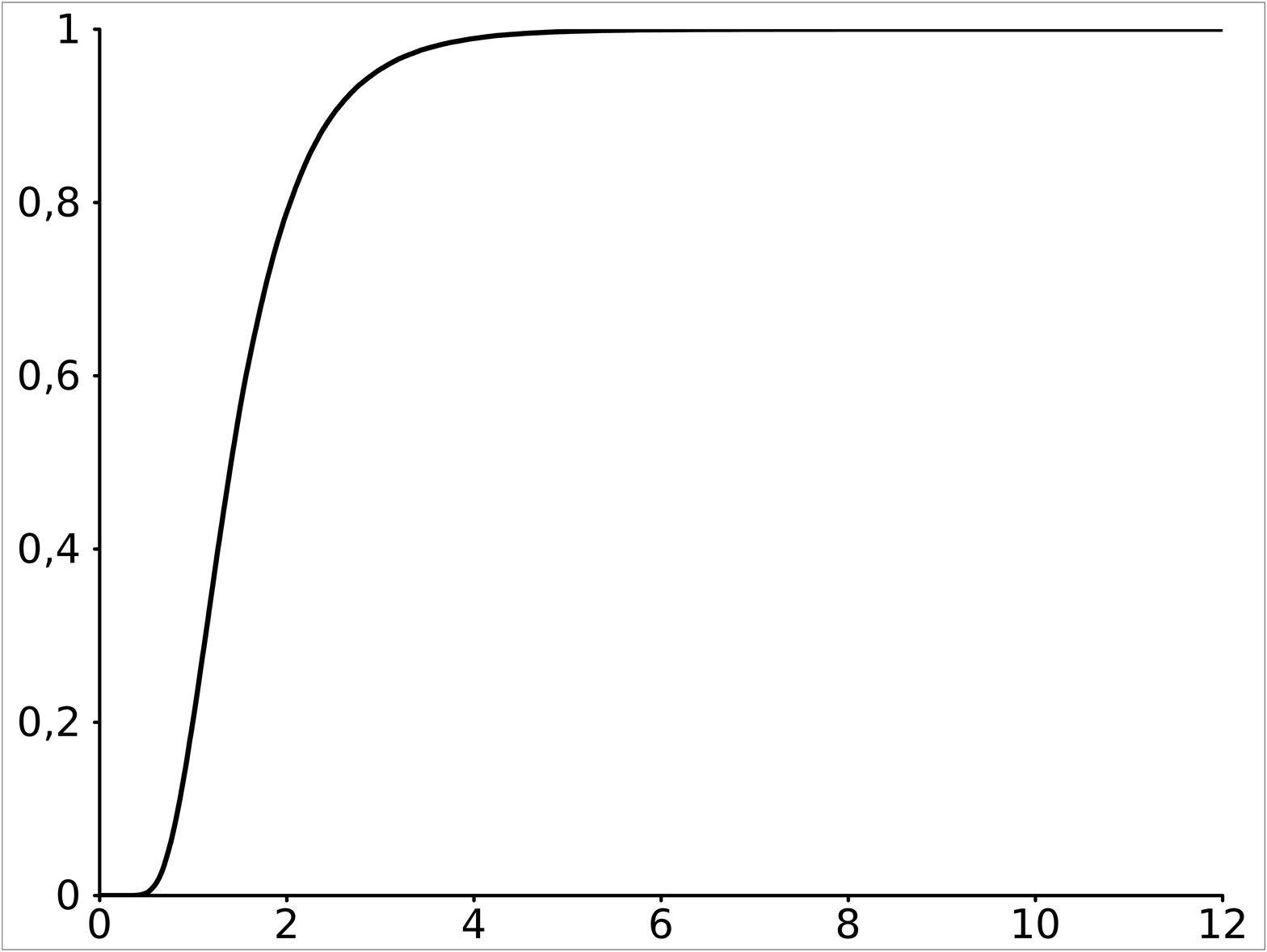}
\end{minipage}
\caption{The empirical distribution function of $W$ (left) and $V$ (right).}
\end{figure}

Next, using different sample sizes from $n=20$ to $n=500$, we simulate
the empirical distribution function of the test statistics $nW_n$ and $nV_n$, respectively.
This was done using $200\,000$ repetitions. 
As shown in Fig. 2, we find that the convergence is very fast overall. 

\begin{figure}[!hbp]
\begin{minipage}[b]{0.45\linewidth} 
\centering
\includegraphics[width=7cm]{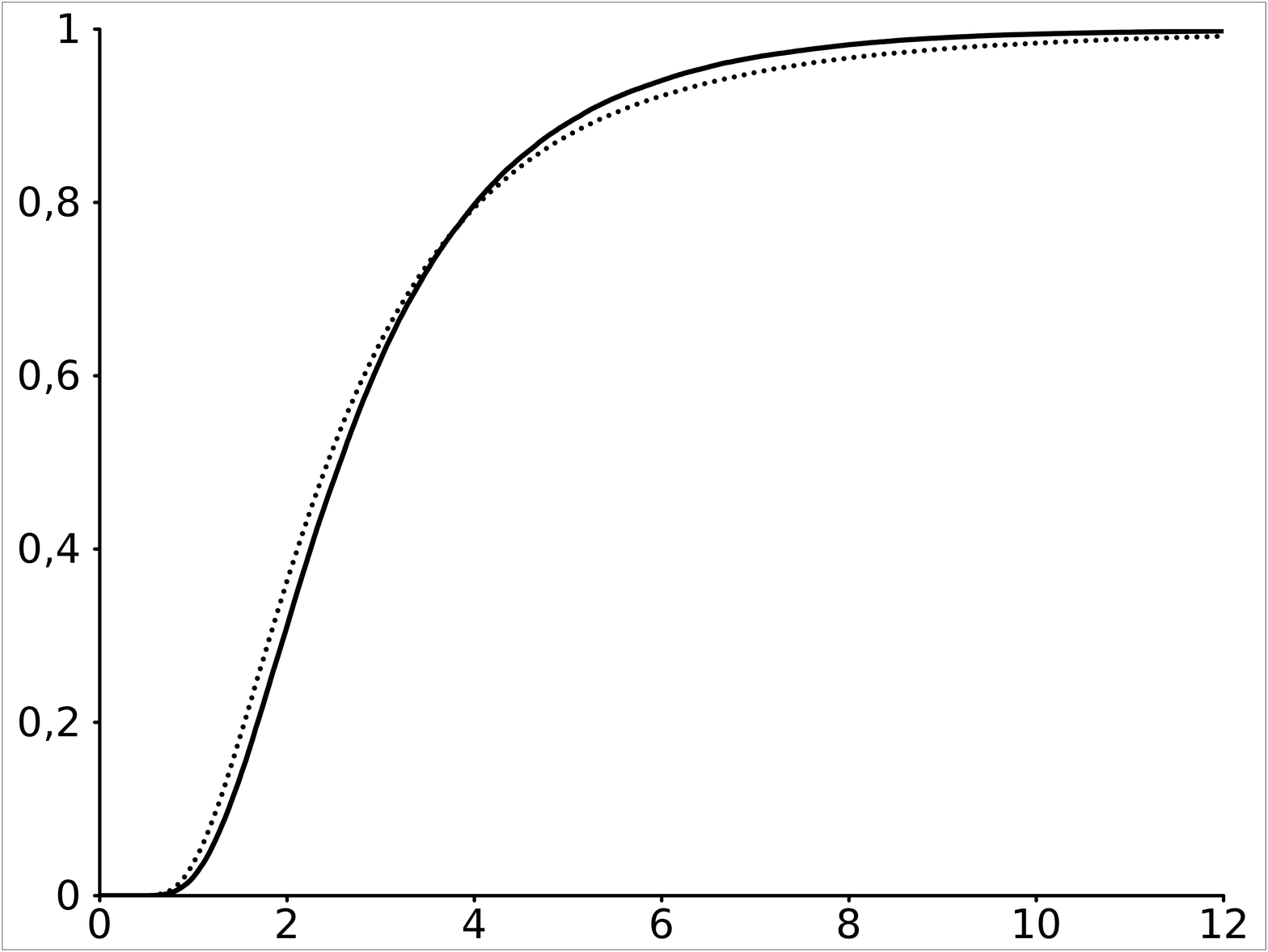}
\end{minipage}
\begin{minipage}[b]{0.45\linewidth} 
\centering
\includegraphics[width=7cm]{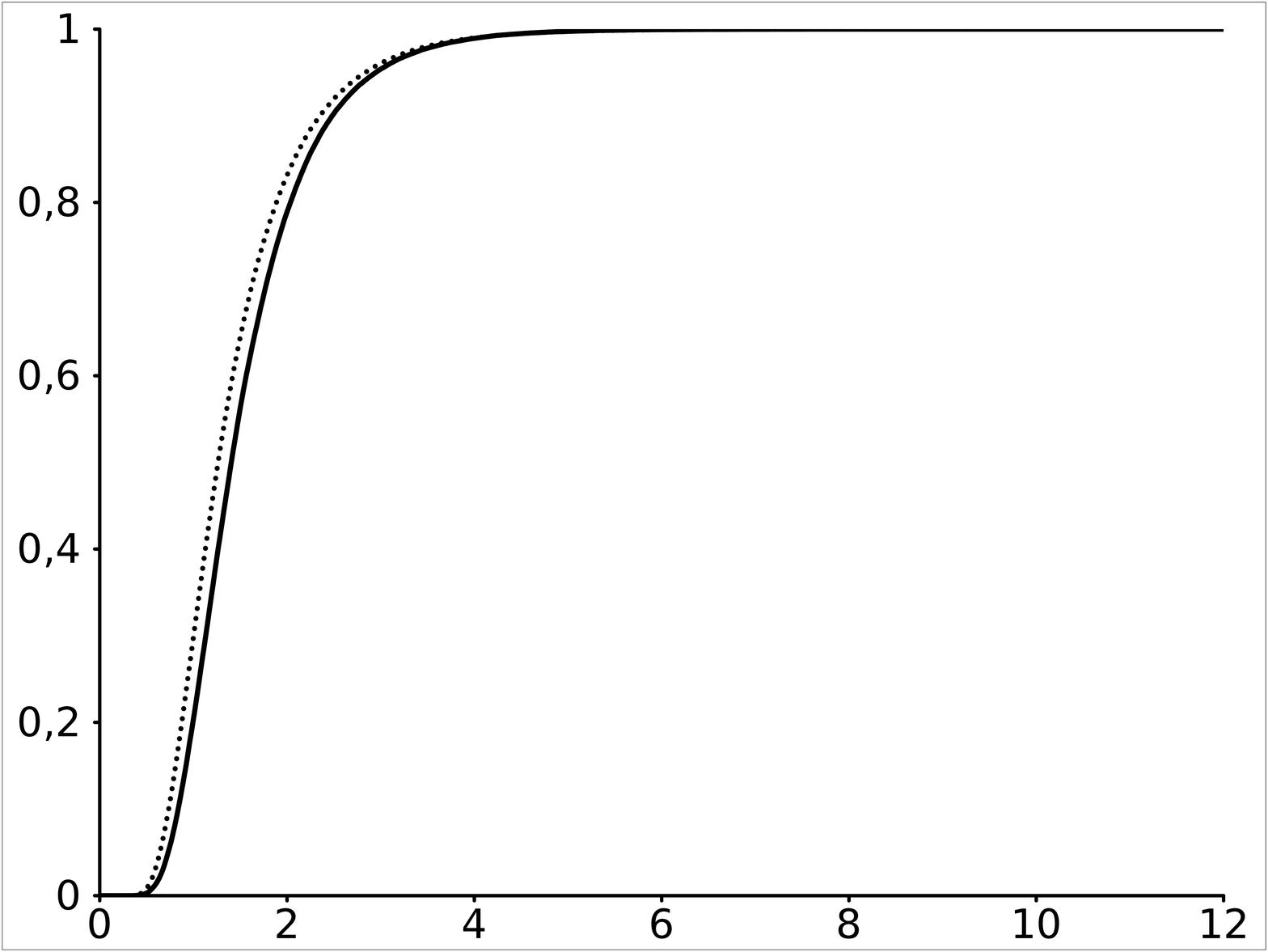}
\end{minipage}
\caption{ The empirical distribution functions of the test statistics
$nW_n$(on the left side) and $nV_n$(on the right side) for $n=20$ (dotted line) and 
the empirical asymptotic distribution functions 
$W$(on the left side) and $V$(on the right side)(thicker line). }
\end{figure}

Table \ref{critical} shows in detail the
empirical critical values of $nW_n$ and $nV_n$ corresponding to the confidence levels $0.85, 0.90, 0.95$
and $0.99$ respectively. The last row corresponding to $n = \infty$ contains
empirical asymptotic critical values for both tests.

\begin{table}[!hbp]
\caption{Critical points of the test statistics $nW_n$ and $nV_n$ for different sample sizes and
different confidence levels.}
\label{critical}      
\begin{tabular}{rrrrrrrrrr}
& & $nW_n$ & & & & &   $nV_n$ & & \\
\hline\noalign{\smallskip}
$ n $   & 0.85 & 0.90 & 0.95 & 0.99 & $ n $ & 0.85 & 0.90 & 0.95 & 0.99 \\
\noalign{\smallskip}\hline\noalign{\smallskip}
20        &  4.60   &  5.43  &  7.00  & 11.40  &     20      &  2.07   &  2.34  &  2.83  &  4.02   \\
50        &  4.52   &  5.25  &  6.66  & 10.76  &     50      &  2.21   &  2.49  &  2.99  &  4.17   \\
100       &  4.49   &  5.20  &  6.50  & 10.40  &    100      &  2.24   &  2.52  &  2.99  &  4.13   \\
200       &  4.48   &  5.15  &  6.39  &  9.87  &    200      &  2.24   &  2.52  &  2.99  &  4.14   \\
500       &  4.47   &  5.13  &  6.31  &  9.39  &    500      &  2.23   &  2.51  &  2.97  &  4.06   \\
$\infty$  &  4.47   &  5.12  &  6.26  &  8.98  &  $\infty$   &  2.22   &  2.49  &  2.95  &  4.02   \\
\noalign{\smallskip}\hline
\noalign{\hrule}
\end{tabular}
\end{table}  

Because of the speed of convergence the asymptotic critical values can be used.
 In the next section we calculate the power of the tests against some alternatives
with finite critical  values due to the similarity of the values.

\subsection{Power of the tests $nV_n$ and $nW_n$}

A simulation study was performed to evaluate the power of the
tests.
In the simulation study we consider some continuous
alternative distributions. All alternative distributions are identified by their
names and are in their standard forms. We give the definition for most
of these distributions. Let $Z$ denote a standard normal random variable.

\begin{enumerate}
\item $Beta(p,q)$  denotes the Beta distribution with density
\hfill\break $f(t)=\Gamma (p+q)t^{p-1}(1-t)^{q-1}/\left( \Gamma (p)\Gamma (q)\right) , \ 0<t<1,\ p,q>0$.
\item The density of  the $Laplace$ distribution is $f(t)=e^{-|t|/2}, \ t\in \mathbb{R}$.
\item The distribution $Lognormal$ denotes the
distribution of the random variable $e^Z$.
\item Two triangle distributions with densities $f(t)=1-|t|, \ -1<t<1$, and
 $f(t)=2-2t, \ 0<t<1$, are denoted, respectively, by $Triangle(I)$ and $Triangle(II)$.
 \item The density of the $Weibull(k)$ distribution is
 $f(t)=kt^{k-1}e^{-t^k}, \ t>0,\ k>0$.
\end{enumerate}

For two tests and sample sizes we use
 the simulated, finite critical points.
 The empirical powers were derived from $200\,000$ simulations
 for all the sample sizes $n=20,50$ and $100$ for two tests. See the details in Table \ref{power}.

 \begin{table}[!hbp]
\caption{Empirical powers (in $\%$) for $nW_n$ and $nV_n$ against some alternatives ($n=20, 50$ and $100$ sample sizes, 
 $*\  100\%$ empirical power, $\alpha $ significance level).}
\label{power}      
\begin{tabular}{l|rrr|rrr||rrr|rrr}
& & $nW_n$ & & & $nW_n$ & & & $nV_n$ & & & $nV_n$ &  \\ 
\noalign{\smallskip}\hline\noalign{\smallskip} 
Alternatives         &  20  &  50 & 100 &  20  &  50 & 100  &  20  &  50 & 100 &  20  & 50& 100\\
\noalign{\smallskip}\hline\noalign{\smallskip}
$N(0,1)$             & 47   &  99  &  *  &  22  &  96  &  *   &  5  &   6 &   8 &   2  &   2 &  4\\
Uniform              &  *   &   *  &  *  &  *   & *    &  *   & 13  &  47 &  93 &   5  &  29 & 82\\
Cauchy               &  88  &  99  &  *  & 84   &  99  &  *   & 88  &  99 &  *  &  84  &  99 &  *\\
Laplace              &  27  &  76  &  97 & 12   &  61  &  93 &  26  &  39 &  55 &  17  &  29 & 43\\
Exp(1)               &  88  &   *  &  *  & 69   &  *   &  *  &  70  &  99 &  *  &  56  &  97 &  *\\
Triangle(I)          &  *   &   *  &  *  &  *   &  *   &  *  &   4  &   7 &  13 &   2  &   3 &  6\\
Triangle(II)         &  *   &   *  &  *  &  *   &  *   &  *  &  21  &  61 &  97 &  11  &  43 & 91\\
Beta(2,2)            &  *   &   *  &  *  &  *   &  *   &  *  &   6  &  15 &  40 &   2  &   7 & 24\\
Weibull(2)           &  *   &   *  &  *  &  *   &  *   &  *  &  12  &  25 &  54 &   5  &  15 & 38\\
Gamma(2,1)           &  25  &  83  &  *  & 10   &  62  &  99 &  40  &  81 &  99 &  27  &  69 & 98\\
Lognormal            &  80  &   *  &  *  & 61   &  *   &  *  &  86  &   * &  *  &  79  &   * &  *\\
Student(5)           &  27  &  82  &  99 & 11   &  67  &  98 &  16  &  19 &  21 &  10  &  12 & 13\\
$\chi^2(1)$          &  88  &   *  &  *  & 71   &  *   &  *  &  94  &   * &  *  &  88  &   * &  *\\
Negativ Exp          &  88  &   *  &  *  & 69   &  *   &  *  &  69  &  99 &  *  &  56  &  97 &  *\\
\noalign{\smallskip}\hline \noalign{\smallskip}
$\alpha$ & & $0.10$ & & & $0.05$ & & & $0.10$ & & & $0.05$ & \\
\noalign{\smallskip}\hline 
\noalign{\hrule}
\end{tabular}
\end{table}
We compare the new test in the location-scale case with Meintanis tests
 based on the empirical characteristic function
 and the empirical momentum generating function from \cite{Me}. 
To the comparison Table 3 from \cite{Me} is used. 
This table next to the power of Meintanis tests contains 
the power of the classical EDF-tests (Kolmogorov-Smirnov, 
Cram\'er-von Mises, Anderson-Darling, Watson) for $n=20$ and $50$ and significance level $\alpha=0.1$. 
In each test in \cite{Me} the location and scale parameter are estimated 
by method of moments or maximum likelihood, 
hereby these tests are adapted to test for composite hypothesis. 
The location and scale test considered in this paper has the greatest power against Cauchy and Laplace alternatives. 
The EDF-tests have greater power against Cauchy and Laplace alternatives than 
Meintanis tests, otherwise the best test is the Meintanis test and our test is the least powerful.

If we test for Logistic location family, we obtain better power than for 
location-scale family, except against gamma, lognormal and $\chi^2(1)$ alternatives. 

A rough general conclusion of this study is that in both cases 
simply computable test statistics are obtained and
 the asymptotic critical values may be used.
For the Logistic location family the test considered is fairly strong,
while for the Logistic location-scale family it seems to be less powerful. 

\section*{Acknowledgements}
The authors are grateful to S. Cs\"org\H o for suggesting the problem and to G. Pap for useful comments and suggestions after carefully reading the manuscript.

\appendix
\section{Proof of Lemma \ref{int_asymp}}
\begin{proof}[Proof of Lemma \ref{int_asymp}] The substitution $x = \frac{t}{1-t}$ yields
\begin{align*}
n\int_{0}^{\frac{1}{n+1}}\ln^k\left(n\frac{t}{1-t}\right)t(1-t)\mathrm{d}t &= n\int_{0}^{\frac{1}{n}}\ln^k(nx)\frac{x}{(1+x)^4}\mathrm{d}x\\
&= \frac{1}{n}\int_{0}^{1}\ln^k y\frac{y}{(1+\frac{y}{n})^4}\mathrm{d}y\\
&= \frac{1}{n}\int_{0}^{1}y \ln^k y\mathrm{d}y +\frac{1}{n}\int_{0}^{1}y\ln^k y\frac{(1-(1+\frac{y}{n})^4)}{(1+\frac{y}{n})^4}\mathrm{d}y\\
&= \frac{1}{n}\int_{0}^{1}y \ln^k y\mathrm{d}y +{\mathcal O}\left(\frac{1}{n^2}\right)\ .
\end{align*}
Since 
\[
\int_{0}^{1}y\mathrm{d}y =\frac{1}{2}
\]
and
\[
\int_{0}^{1}y \ln^{k} y\mathrm{d}y = \lim_{\e \to 0}\left[\frac{y^2}{2}\ln^k y\right]_{\e}^{1}-\frac{k}{2}\int_{0}^{1}y\ln^{k-1}y\mathrm{d}y 
= -\frac{k}{2}\int_{0}^{1}y\ln^{k-1}y\mathrm{d}y\ ,
\]
the exact form of the leading coefficient follows.
\end{proof}

\section{Proof of Lemma \ref{jacobi_integrals}}

The Jacobi polynomials with parameters $(1,1)$ have the following generating function:
\begin{equation}
\label{gen_f}
\sum_{n=0}^{\infty} P_{n}^{(1,1)}(x)z^n = \frac{4}{R(1-z+R)(1+z+R)}\ ,
\end{equation}
where
$$R=\sqrt{1-2zx+z^2}$$
(see \cite{abramowitz_stegun}, Table 22.9).
For fixed $|z|<1$ the only branch point of the square root $R$, as a function of $x$, is located at
\begin{equation}
x_0 = \frac{1}{2}\left(z+\frac{1}{z}\right)\ .
\end{equation}
From elementary conformal mapping it is obvious that $x_0 \in {\mathbb C}\setminus[-1,1]$ as long as $|z| <1$. Therefore there is a unique choice of the branch of $R$ on $[-1,1]$ such that
\begin{equation}
\label{branchR}
\sqrt{1-2zx+z^2}\Big|_{x=1} = 1-z\ .
\end{equation}

\begin{proof}[Proof of Lemma \ref{jacobi_integrals}]
Consider the function
\begin{equation}
f(x) = (1-x^2)\ln\left(\frac{1+x}{1-x}\right) \qquad -1 \leq x \leq 1\ .
\end{equation}
It is easy to see that $|f(x)| <1$ on $[-1,1]$ (the exact upper bound is irrelevant for our purposes).
Consider the integrals
\begin{equation}
a_n =\int_{-1}^{1} P_n^{(1,1)}(x)f(x)\mathrm{d}x \qquad n=0,1,\dots
\end{equation}
and their (formal) generating function
\begin{equation}
g(z) = \sum_{n=0}^{\infty}a_n z^n\ .
\end{equation}
Since
$$
\left|P_n^{(1,1)}(x)\right| \leq n+1 \qquad -1 \leq x \leq 1\ ,
$$
we have $|a_n| \leq 2(n+1)$ and therefore the power series of $g(z)$ converges absolutely and uniformly in the interior of the unit disk $|z|<1$. 
Therefore, for  any fixed $|z|<1$,
\begin{align*}
g(z) = \sum_{n=0}^{\infty}\int_{-1}^{1} P_n^{(1,1)}(x)z^nf(x) \mathrm{d}x &= \int_{-1}^{1}\sum_{n=0}^{\infty}P_n^{(1,1)}(x)z^nf(x) \mathrm{d}x\\
&=\int_{-1}^{1}\frac{4}{R(1-z+R)(1+z+R)}f(x)\mathrm{d}x\ .
\end{align*}
In the last step we used the generating function identity \eqref{gen_f} valid for $|z|<1$.

Assume now that $z$ is real and $0< z <1$. The integral above can be calculated explicitly by using the Euler substitution $u =\sqrt{1-2zx+z^2}$\ :
$$
\left\{
\begin{split}
x &= \frac{z^2+1-u^2}{2z}\\
dx &= -\frac{u}{z}du\\
u_1 & = 1+z\\ 
u_2 & = 1-z\\
\end{split}
\right.
$$
The mapping $x=x(u)$ is strictly decreasing from the interval $[1-z,1+z]$ onto $[-1,1]$.
Simple algebraic manipulations yield
\begin{align*}
&\int_{-1}^{1}\frac{4(1-x^2)}{R(1-z+R)(1+z+R)}\log\left(\frac{1+x}{1-x}\right)\mathrm{d}x\\
&=\int_{1-z}^{1+z}\frac{(u+z-1)(z+1-u)}{z^3}\log\left(\frac{(z+1+u)(z+1-u)}{(u-z+1)(u+z-1)}\right)\mathrm{d}u\\
&=\lim_{\e \to 0}\left[\frac{z^2(u-1)-\frac{1}{3}(u-1)^3}{z^3}\log\left(\frac{(z+1+u)(z+1-u)}{(u-z+1)(u+z-1)}\right)\right]_{1-z+\e}^{1+z-\e}\\
&\quad -\int_{1-z}^{1+z}\frac{z^2(u-1)-\frac{1}{3}(u-1)^3}{z^3}\frac{8uz}{(u^2-(z+1)^2)(u^2-(1-z)^2)}\mathrm{d}u\ .
\end{align*}

The first term gives
\begin{align*}
&\left[\frac{z^2(u-1)-\frac{1}{3}(u-1)^3}{z^3}\log\left(\frac{(z+1+u)(z+1-u)}{(u-z+1)(u+z-1)}\right)\right]_{1-z+\e}^{1+z-\e}\\
\ &= \frac{z^2(z-\e)-\frac{1}{3}(z-\e)^3}{z^3}\left(\log\left(\frac{(2+2z-\e)\e}{(2-\e)(2z-\e)}\right)+\log\left(\frac{(2+\e)(2z-\e)}{(2-2z+\e)\e}\right)\right)\\
\ &= \frac{z^2(z-\e)-\frac{1}{3}(z-\e)^3}{z^3}\log\left(\frac{(2+\e)(2+2z-\e)}{(2-\e)(2-2z+\e)}\right) \to \frac{2}{3}\ln\left(\frac{1+z}{1-z}\right)
\end{align*}

as $\e \to 0$. The integrand in the second term has the partial fraction decomposition
\begin{align*}
&-\frac{8}{3z^2} -\frac{2z^3-12z}{3z^3}\left(\frac{1}{u+z+1} +\frac{1}{u-z+1}\right)\\
&+\frac{8}{3z^3}\left(\frac{1}{u+z+1} -\frac{1}{u-z+1}\right) +\frac{2}{3}\left(\frac{1}{u+z-1}+\frac{1}{u-z-1}\right)
\end{align*}

Therefore,
\begin{align*}
&-\int_{1-z}^{1+z}\frac{z^2(u-1)-\frac{1}{3}(u-1)^3}{z^3}\frac{8uz}{(u^2-(z+1)^2)(u^2-(z-1)^2)}\mathrm{d}u\\
\quad &=
\lim_{\e \to 0}\left[\frac{8}{3z^2}u +\frac{2z^3-12z}{3z^3}\log((u+z+1)(u-z+1))\right.\\
&\left. -\frac{8}{3z^3}\log\left(\frac{u+z+1}{u-z+1}\right)-\frac{2}{3}\log((u+z-1)(u-z-1)) \right]_{1-z+\e}^{1+z-\e}\\
&=\frac{16}{3z} +\frac{2z^3-12z}{3z^3}\log\left(\frac{1+z}{1-z}\right) -\frac{8}{3z^3}\log\left((1+z)(1-z)\right)\ .
\end{align*}

Combining the two expressions above we get
\begin{align}
&\int_{-1}^{1}\frac{4}{R(1-z+R)(1+z+R)}f(x)\mathrm{d}x\\
&= \frac{4}{3z^3}\left[4z^2 +(z^3-3z)\log\left(\frac{1+z}{1-z}\right) -2\log\left((1+z)(1-z)\right)\right]\ ,
\end{align}
for $z \in (0,1)$. With the proper choice of the branches of the logarithms this represents a holomorphic function on the punctured disk $0<|z|<1$.
The Laurent series expansion of the right hand side at $z=0$ can be written down explicitly:
\begin{equation}
\frac{4}{3z^3}\left[4z^2 +(z^3-3z)\log\left(\frac{1+z}{1-z}\right) -2\log\left((1+z)(1-z)\right)\right] = 8\sum_{k=0}^{\infty}\frac{z^{2k+1}}{(2k+1)(2k+3)(k+2)} 
\end{equation}
Thus this function has a removable singularity at $z=0$ and it coincides with $g(z)$ on the interval $(0,1)$. Therefore
\begin{equation}
g(z)=8\sum_{k=0}^{\infty}\frac{z^{2k+1}}{(2k+1)(2k+3)(k+2)}\ , 
\end{equation}
which implies that $a_{2k}=0$ and
\begin{equation}
a_{2k+1}= \frac{8}{(2k+1)(2k+3)(k+2)}\ .
\end{equation}
\end{proof}

 \bibliographystyle{abbrv}
 \bibliography{weighted_correlation_logistic}
\end{document}